\documentclass[reqno,twoside,12pt]{amsart}

\hoffset - 1.5cm

\usepackage{amsmath}
\usepackage{amssymb}
\usepackage{graphicx}
\usepackage{amstext}
\usepackage{amsopn}
\usepackage{amsfonts}
\usepackage{xcolor}
\usepackage{setspace}
\usepackage{enumerate}
\usepackage[polish,english]{babel}
\usepackage[utf8]{inputenc}
\usepackage{polski}

\usepackage{lmodern}

\newtheorem{Remark}{Remark}[section]
\newtheorem{Corollary}[Remark]{Corollary}
\newtheorem{Definition}[Remark]{Definition}
\newtheorem{Example}[Remark]{Example}
\newtheorem{Fact}[Remark]{Fact}
\newtheorem{Lemma}[Remark]{Lemma}
\newtheorem{Proposition}[Remark]{Proposition}
\newtheorem{Theorem}[Remark]{Theorem}

\newcommand{\ba}{\begin{array}}
\newcommand{\bc}{\begin{center}}
\newcommand{\bd}{\begin{description}}
\newcommand{\bdm}{\begin{displaymath}}
\newcommand{\be}{\begin{enumerate}}
\newcommand{\beq}{\begin{equation}}
\newcommand{\bdf}{\begin{Definition}}
\newcommand{\bex}{\begin{Example}}
\newcommand{\bft}{\begin{Fact}}
\newcommand{\bl}{\begin{Lemma}}
\newcommand{\bp}{\begin{Proposition}}
\newcommand{\br}{\begin{Remark}}
\newcommand{\bt}{\begin{Theorem}}
\newcommand{\bco}{\begin{Corollary}}
\newcommand{\bhy}{\begin{Hypothesis}}
\newcommand{\ea}{\end{array}}
\newcommand{\ec}{\end{center}}
\newcommand{\ed}{\end{description}}
\newcommand{\edm}{\end{displaymath}}
\newcommand{\ee}{\end{enumerate}}
\newcommand{\eeq}{\end{equation}}
\newcommand{\edf}{\end{Definition}}
\newcommand{\eex}{\end{Example}}
\newcommand{\eft}{\end{Fact}}
\newcommand{\el}{\end{Lemma}}
\newcommand{\ep}{\end{Proposition}}
\newcommand{\er}{\end{Remark}}
\newcommand{\et}{\end{Theorem}}
\newcommand{\eco}{\end{Corollary}}
\newcommand{\ehy}{\end{Hypothesis}}

\newcommand{\bH}{\mathbb{H}}
\newcommand{\bI}{\mathbb{I}}

\newcommand{\bR}{\mathbb{R}}

\newcommand{\bT}{\mathbb{T}}

\newcommand{\bV}{\mathbb{V}}
\newcommand{\bW}{\mathbb{W}}

\newcommand{\bZ}{\mathbb{Z}}

\newcommand{\cC}{\mathcal{C}}

\newcommand{\cL}{\mathcal{L}}

\newcommand{\cN}{\mathcal{N}}
\newcommand{\cO}{\mathcal{O}}

\newcommand{\cR}{\mathcal{R}}

\newcommand{\cT}{\mathcal{T}}

\newcommand{\cV}{\mathcal{V}}
\newcommand{\cW}{\mathcal{W}}

\numberwithin{equation}{section} \errorcontextlines=0

\newcommand{\codim}{\mathrm{ codim  }}

\newcommand{\sub}{\overline{\mathrm{sub}}}

\newcommand{\bif}{\mathcal{BIF}}

\makeatletter
\@namedef{subjclassname@2020}{%
  \textup{2020} Mathematics Subject Classification}
\makeatother

\usepackage{todonotes}
\usepackage{setspace}

\begin{document}

\title[Global bifurcation for elliptic systems with symmetries]{Global bifurcation of solutions to elliptic systems with system and domain symmetries}

\author{Piotr Stefaniak}
\address{Faculty of Mathematics and Computer Science \\ Nicolaus Copernicus University in Toru\'n\\
PL-87-100 Toru\'{n} \\ ul. Chopina $12 \slash 18$ \\
Poland}

\email{cstefan@mat.umk.pl (P. Stefaniak)}

\date{\today}

\keywords{global bifurcation, symmetric elliptic systems, equivariant degree}
\subjclass[2020]{Primary 35B32; Secondary 35J47, 35B06, 58J05}

\begin{abstract}
We study parameterized elliptic systems on symmetric domains with additional system symmetries. We prove the existence of continua of nontrivial solutions bifurcating from the constant branch determined by a critical point of the potential, without assuming nondegeneracy, via the degree for equivariant gradient maps. Our assumptions are formulated in terms of the right–hand side. When the domain is a compact symmetric space, the bifurcating solutions break symmetry at every nonzero level. Under additional assumptions on the right–hand side, the continua are unbounded.
\end{abstract}

\maketitle

\section{Introduction}
We study global bifurcations of weak solutions to nonlinear elliptic systems on domains with symmetries, where the systems themselves also carry symmetries. For a compact Riemannian manifold $M$ (possibly with boundary) and a potential $F:\bR^{p}\times\bR\to\bR$, we consider
\begin{equation}\label{eq:sys-intr}
-\Delta_{M}u=\nabla_{u}F(u,\lambda),
\end{equation}
where $\lambda\in\bR$, with Neumann boundary conditions when $\partial M\ne\emptyset$. Here $M$ is the domain, while $\bR^p$ is the value space of the unknown $u$. The precise assumptions used throughout the paper are collected in Subsection~\ref{subsec:system}.

Bifurcation theory provides tools for detecting nontrivial solutions which emanate from a branch of trivial solutions. In elliptic problems, local and global bifurcation results have been developed in several directions, including bifurcation from simple eigenvalues and, separately, global continuation results based on topological degree; see, for example, \cite{CrandallRabinowitz,Nirenberg,Rabinowitz1971,Rabinowitz1986}. The global point of view gives information on the continuation of bifurcating branches away from a neighbourhood of the bifurcation point, by detecting connected sets of nontrivial solutions. Bifurcation phenomena in symmetric problems, including symmetry breaking, have also been studied from several points of view. Classical symmetry-breaking results were obtained, for example, by Dancer and by Smoller and Wasserman; see \cite{Dancer1979,Dancer1982,SmollerWasserman}.

For the approach followed here, the presence of symmetries leads to a specific limitation of the ordinary degree. For an equivariant completely continuous perturbation of the identity, the Leray--Schauder degree is determined by the restriction to the fixed-point subspace, see, for example, \cite{Dancer1982}. Thus changes in non-fixed directions may be invisible. To keep track of such changes one has to use invariants which take the group action into account. In the gradient setting, finite-dimensional constructions of the equivariant gradient degree are discussed in \cite{Geba}, while its infinite-dimensional version has been developed in \cite{Ryb2005milano}. For related approaches to bifurcation based on equivariant Conley index methods, see for example \cite{GolRyb2013,Izydorek}. A different, but related, invariant used in bifurcation theory is equivariant spectral flow, see \cite{IJSW2025}.

We now turn to the variational and equivariant structure of the present problem. We assume that, for every $\lambda$, the potential $F(\cdot,\lambda)$ has a critical point. After a translation we take this critical point to be $0$, and hence \eqref{eq:sys-intr} has a branch of trivial constant solutions. The problem is variational: under our assumptions, weak solutions of \eqref{eq:sys-intr} are critical points of a $C^2$ functional $\Phi:\bH\times\bR\to\bR$ defined on a Hilbert space $\bH$, and $\nabla_u\Phi$ is a completely continuous perturbation of the identity.

In the problem considered here there are two sources of symmetry. A compact Lie group acts on the value space $\bR^p$ through the symmetries of the potential, while another compact Lie group acts on the domain $M$. These symmetries are inherited by both $\bH$ and $\Phi$, and the induced action on the function space is a product action. We therefore work with the degree for equivariant gradient maps, whose values lie in the Euler ring of the acting product group.

The critical point on the trivial branch is not assumed to be nondegenerate. Here nondegeneracy means that the Hessian $\nabla_u^2F(0,\lambda)$ is nonsingular. Instead, we assume nontriviality of the equivariant gradient degree of $\nabla_uF(\cdot,\lambda)$ near $0$. This includes the standard nondegenerate case and the case of nonzero Brouwer degree, but also allows situations in which the ordinary Brouwer degree vanishes while the equivariant degree remains nontrivial; see Remark~\ref{rem:B6}.

As a first step in the bifurcation analysis, we identify the admissible parameter values, namely those at which bifurcation from the trivial branch may occur; see Lemma~\ref{lem:nec}. In the nondegenerate case this follows from the implicit function theorem, while in the degenerate case it is obtained from the equivariant structure of the linearized problem. This gives the candidate set of bifurcation parameters.

To obtain bifurcation results, we introduce an equivariant bifurcation index defined by means of the degree for equivariant gradient maps. Nontriviality of this index implies global bifurcation; see Theorem~\ref{thm:bif-index-global}. The main point is to prove that the index is nonzero in the present setting. Its computation requires working in the Euler ring of the acting product group. This is, in general, difficult to handle, since the multiplicative structure of such an Euler ring is not fully known. For tori, however, the situation is much better: explicit multiplication rules are available, see \cite{GarRyb}. We therefore pass to maximal tori in the symmetry groups of the domain and of the system and carry out all computations in the Euler ring of their product, which is again a torus. This reduction is used to prove nontriviality of the bifurcation index. In general, it does not keep all the isotropy information of the original group action, but it gives the information needed for the global bifurcation argument. The key tools in the computation are the tensor product of representations (Lemma~\ref{lem:external-tensor}) and its decomposition into torus representations (Lemma~\ref{lem:T-product-decomp}).

The resulting global bifurcation theorems are Theorems~\ref{thm:global-bifurcation} and~\ref{thm:global-bifurcation-trivial}. If the necessary condition of Lemma~\ref{lem:nec} is not satisfied, Theorem~\ref{thm:local-or-global} gives a local--global alternative: either local bifurcation occurs near the parameter in question, or a change of degree forces global bifurcation at an intermediate parameter value.

We next turn our attention to compact symmetric spaces $M=G/H$. Using the structure of Laplace--Beltrami eigenspaces as torus representations, we show that global bifurcation occurs at every candidate parameter value, with the possible exception of $\lambda=0$. At each nonzero level the bifurcating solutions break the domain symmetry: their $G$-isotropy group is strictly smaller than $G$; see Theorem~\ref{thm:global-bifurcation-symmetric}. In particular, for $G=\mathrm{SO}(n)$ and $H=\mathrm{SO}(n-1)$, so that $M=S^{n-1}$, this yields bifurcation of nonradial solutions from the radial (constant) ones.

The last result concerns the structure of the bifurcating continua. In Theorem~\ref{thm:unbounded} we prove that, under assumptions on the eigenspaces of the Hessian of $F$ and with $0$ nondegenerate, every bifurcating continuum is unbounded. The proof uses Lemma~\ref{lem:highest} to identify a torus weight which appears at a given bifurcation level and not at the preceding ones. This allows us to find a nonzero coefficient in the bifurcation index which cannot be cancelled in the summation formula from the Rabinowitz alternative (Theorem~\ref{thm:altRab_sum_T}).

After this introduction, the article is organized as follows. In Section~\ref{reptor} we recall basic notions of representation theory for compact Lie groups, then restrict to torus actions and record the decomposition rules needed later. Subsection~\ref{subsec:Eringtorus} reviews the Euler ring of a torus and establishes several algebraic identities used in the bifurcation-index computations. Subsection~\ref{subsec:degree} summarizes the degree for equivariant gradient maps and its properties.
In Subsection~\ref{subsec:system} we introduce the elliptic system \eqref{eq:GH}--\eqref{eq:GH-Neumann}, list the standing assumptions \textnormal{(B1)}--\textnormal{(B6)} and set up the variational and symmetry framework. In Subsection~\ref{subsec:global} we analyze the bifurcation problem, deriving the necessary condition and proving sufficient conditions. In Subsection~\ref{subsec:symmetry-breaking} we establish symmetry breaking for the bifurcating solutions. Finally, in Subsection~\ref{subsec:unbounded} we specialize to compact Riemannian symmetric spaces, where we prove global bifurcation with symmetry breaking and the unboundedness of the bifurcating continua.

\section{Preliminaries}\label{sec:prelim}

In this section we collect the basic tools used later in the paper. Most of the material is standard, but we also include several auxiliary lemmas tailored to our setting: in particular, we describe the decomposition of tensor products of torus representations and establish algebraic identities in the Euler ring that will be used in the bifurcation index computations.

\subsection{Elements of representation theory}\label{reptor}

Let $G$ be a compact Lie group. An orthogonal $G$-representation is a pair $(\bV,\rho)$ with $\bV$ a finite-dimensional real vector space and $\rho\colon G\to O(\bV)$ a continuous homomorphism. We do not distinguish between the representation and its underlying space, writing simply $\bV$. When clear, we denote the action either by $gv$ or by $g\cdot v$, with $gv=g\cdot v=\rho(g)v$ for $g\in G$, $v\in\bV$. A $G$-representation is called irreducible if it has no nonzero proper $G$-invariant subspace (i.e., the only $G$-invariant subspaces are $\{0\}$ and $\bV$).

Two $G$-representations $(\bV,\rho_1)$ and $(\bW,\rho_2)$ are equivalent if there exists a linear isomorphism $T:\bV\to\bW$ which is $G$-equivariant, i.e. such that
$T\circ\rho_1(g)=\rho_2(g)\circ T$ for all $g\in G$.
We denote this by $\bV\approx_G\bW$.

Let $\Gamma$ and $G$ be compact Lie groups. Let $(\bW,\rho_1)$ be a finite-dimensional orthogonal $\Gamma$-representation and $(\bV,\rho_2)$ a finite-dimensional orthogonal $G$-representation.
 Put $N=\dim\bW$ and equip the orthogonal direct sum $\bigoplus_{i=1}^{N}\bV$ with the $(\Gamma\times G)$–action
\[
(\gamma,g)\cdot(v_{1},\dots,v_{N})
:=\rho_{1}(\gamma)\bigl(\rho_{2}(g)v_{1},\dots,\rho_{2}(g)v_{N}\bigr),
\]
i.e. apply $\rho_2(g)$ componentwise, then let
$\rho_1(\gamma)$ act on the list of components via its action on $\bW$.
On the tensor product $\bW\otimes\bV$ we take the product action
\[
(\gamma,g)\cdot(w\otimes v):=(\rho_{1}(\gamma)w)\otimes(\rho_{2}(g)v).
\]

\begin{Lemma}\label{lem:external-tensor}
With the above notation, the orthogonal $\Gamma\times G$–representations $\bigoplus_{j=1}^{N}\bV$ and $\bW\otimes\bV$ are equivalent.
\end{Lemma}

\begin{proof}
Fix an orthonormal basis $(w_{1},\dots,w_{N})$ of $\bW$ and define an isomorphism
\[
T: \bW\otimes\bV\longrightarrow \bigoplus_{j=1}^{N}\bV,\qquad
T(w_{j}\otimes v):=(0,\dots,0,\ \text{$v$ in the $j$-th slot}\ ,0,\dots,0),
\]
extended linearly. From the definition of the $\Gamma\times G$–actions,
it easily follows that $T$ is $\Gamma\times G$–equivariant.
\end{proof}

\subsubsection{Torus representations}
We now restrict attention to torus actions, which will be used in the sequel.
Let $T^{r}$ be a torus of rank $r$. For $m\in\mathbb{Z}^{r}$ denote by $\bR[1,m]$ the real two--dimensional orthogonal $T^{r}$--representation in which
$e^{i\phi}=(e^{i\phi_1},\dots,e^{i\phi_r})\in T^r$ acts on $\mathbb{R}^{2}$ as the rotation through the angle $\langle m,\phi\rangle$, where $\phi=(\phi_1,\dots,\phi_r)\in\mathbb{R}^{r}$:
\[
\bR[1,m]:=\Bigl(\mathbb{R}^{2},\ \phi\mapsto
\begin{pmatrix}
\cos\langle m,\phi\rangle & -\sin\langle m,\phi\rangle\\
\sin\langle m,\phi\rangle & \ \cos\langle m,\phi\rangle
\end{pmatrix}\Bigr).
\]
For $k\in\mathbb{N}$ put $\bR[k,m]:=\bigoplus_{j=1}^{k}\bR[1,m]$. In the case $m=0$ we write $\bR[k,0]:=\mathbb{R}^{k}$ with the trivial action.

As is standard (see, for example, \cite[Proposition~II.8.5]{BrockerDieck}), every real irreducible $T^r$–rep\-re\-sen\-tation is equivalent to $\bR[1,m]$ for some $m\in\mathbb{Z}^r$ and $\bR[1,m]\approx_{T^r}\bR[1,m']$ precisely when $m'=\pm m$.
Consequently, every finite–dimensional orthogonal $T^r$–representation decomposes as follows.

\begin{Corollary}\label{cor:repT}
Let $\bV$ be a finite–dimensional orthogonal $T^r$-representation. Then there exist $m_1,\dots,m_l\in \mathbb{Z}^r\setminus\{0\}$ and integers $l\ge 0$, $k_0\ge 0$, $k_1,\dots,k_l\ge 1$ such that
\[
\bV \approx_{T^r} \bR[k_0,0] \oplus \bR[k_1,m_1] \oplus \cdots \oplus \bR[k_l,m_l].
\]
\end{Corollary}

For  $m=(m_1,\dots,m_r)\in\bZ^{r}$ and $n=(n_1,\dots,n_l)\in\bZ^{l}$, by $(m,n)$ we understand $(m_1,\dots,m_r,n_1,\dots,n_l)\in\bZ^{r+l}$. This convention will be used throughout the paper.

The following lemma records the $T^{r+l}$–decomposition of a tensor product of torus representations.

\begin{Lemma}\label{lem:T-product-decomp}
Let $\bW$ be a $T^{r}$–representation and $\bV$ a $T^{l}$–representation with
\[
\bW\approx_{T^r} \bR[l_{0},0]\oplus \bigoplus_{m\in\bZ^{r}\setminus\{0\}} \bR[l_{m},m],
\qquad
\bV\approx_{T^l} \bR[k_{0},0]\oplus \bigoplus_{n\in\bZ^{l}\setminus\{0\}} \bR[k_{n},n].
\]
Then
\begin{eqnarray*}
\bW\otimes\bV &\approx_{T^{r+l}}&\bR[k_{0}l_{0},(0,0)]
\oplus \bigoplus_{m\in\bZ^{r}\setminus\{0\}}\bR[k_{0}l_{m},(m,0)]
\oplus \bigoplus_{n\in\bZ^{l}\setminus\{0\}}\bR[k_{n}l_{0},(0,n)]
\\
&&\oplus \bigoplus_{m\in\bZ^{r}\setminus\{0\},  n\in\bZ^{l}\setminus\{0\}}
\bigl(\bR[k_{n}l_{m},(m,n)]\oplus \bR[k_{n}l_{m},(m,-n)]\bigr).
\end{eqnarray*}
\end{Lemma}

\begin{proof}
By distributivity of $\otimes$ it suffices to show, for $m\in\bZ^{r}$ and $n\in\bZ^{l}$,
\[
\bR[1,m]\otimes \bR[1,n]\approx_{T^{r+l}} \begin{cases}
\bR[1,(m,n)]\oplus \bR[1,(m,-n)] & m\ne 0,\ n\ne 0,\\
\bR[1,(m,0)]                     & n=0,\ m\ne 0,\\
\bR[1,(0,n)]                     & m=0,\ n\ne 0,\\
\bR[1,0]                         & m=0,\ n=0.
\end{cases}
\]
For $m\ne0$, $n\ne0$, define $T^{r+l}$–equivariant maps
$$
T_\pm: \bR[1,(m,\pm n)]\longrightarrow \bR[1,m]\otimes \bR[1,n]
$$
by
\[
\begin{aligned}
&T_+(e_1):=e_1\otimes e_1-e_2\otimes e_2,\quad &&T_+(e_2):=e_1\otimes e_2+e_2\otimes e_1,\\
&T_-(e_1):=e_1\otimes e_1+e_2\otimes e_2,\quad &&T_-(e_2):=e_1\otimes e_2-e_2\otimes e_1.
\end{aligned}
\]
where $e_1,e_2$ is the standard basis of $\bR^2$. Their images are $T^{r+l}$–invariant, orthogonal, and
the direct sum of their ranges equals $\bR[1,m]\otimes \bR[1,n]$,
giving the first case. The remaining cases are immediate. Collecting multiplicities in the distributive expansion yields the stated decomposition.
\end{proof}

\subsection{Euler ring of a torus}\label{subsec:Eringtorus}

Let $\sub(T^{r})$ be the set of closed subgroups of $T^{r}$. For $H\in\sub(T^{r})$ let $T^{r}/H$ denote the orbit space, and write $T^{r}/H^{+}$ for the pointed space obtained by adjoining a disjoint base point. Let $\chi_{T^{r}}(X)$ denote the $T^{r}$–equivariant Euler characteristic of a $T^{r}$–CW complex $X$. The Euler ring of the torus $U(T^{r})$ is the free $\mathbb{Z}$–module generated by the elements $\chi_{T^{r}}(T^{r}/H^{+})$ with $H\in\sub(T^{r})$. The multiplication in $U(T^r)$ is induced by the smash product of pointed $T^r$–CW complexes.
For tori this product is explicit (see \cite{GarRyb}): if $H,H'\in\sub(T^r)$ and $H'':=H\cap H'$, then
\begin{equation}\label{eq:UTmultiplication}
\chi_{T^r}(T^r/H^{+})\star\chi_{T^r}(T^r/H'^{+})=
\begin{cases}
\chi_{T^r}(T^r/H''^{+}), & \dim H+\dim H'=\dim T^r+\dim H'',\\[2pt]
\Theta, & \dim H+\dim H'<\dim T^r+\dim H'',
\end{cases}
\end{equation}
with unit $\bI:=\chi_{T^r}\bigl((T^r/T^r)^{+}\bigr)$ and zero element $\Theta$.

Directly from \eqref{eq:UTmultiplication} we get:

\begin{Corollary}\label{cor:bigcodim}
Let $X,Y\in U(T^r)$ with $Y$ a $\mathbb{Z}$–linear combination of terms $\chi_{T^r}(T^r/H^{+})$ with $\dim H\le r-2$.
Then $X\star Y$ is again a $\mathbb{Z}$–linear combination of terms $\chi_{T^r}(T^r/H^{+})$ with $\dim H\le r-2$.
\end{Corollary}

For $k\in\bZ^{r}$ define
$
H_k:= \{e^{i\phi}\in T^r: e^{i\langle k,\phi\rangle}=1 \}.
$
Arguing as in \cite[Remark 2.9]{GolRybSte2025}, every $H\in\sub(T^r)\setminus\{T^r\}$ is an intersection of codimension-$1$ subgroups.
In the next lemma we show that any $H\in\sub(T^r)\setminus\{T^r\}$ can be written as the intersection of exactly $\codim H$ such subgroups $H_k$.

\begin{Lemma}\label{lem:codimH}
Let $H \in \sub(T^r)$ be of codimension $s$. Then there exist
$k_1,\dots,k_s \in \mathbb{Z}^r\setminus\{0\}$ such that
$H=H_{k_1}\cap\ldots\cap H_{k_s}.$
\end{Lemma}

\begin{proof}
Let
$H=H_{l_1}\cap\ldots\cap H_{l_p}$, where $l_1,\ldots,l_p\in \mathbb{Z}^r\setminus\{0\}$.
Let $R \in M_{p\times r}(\mathbb{Z})$ have rows $l_1,\dots,l_p$. Then
$
H=\{e^{i\phi}\in T^r : R\phi \in (2\pi\mathbb{Z})^{p}\}.
$
By the Smith normal form, see, for example, \cite[Theorem 6.6]{Matthews}, there exist unimodular $P \in \mathrm{GL}_p(\mathbb{Z})$ and
$Q \in \mathrm{GL}_r(\mathbb{Z})$ such that
\[
D=PRQ=(D_{ij})_{1\le i\le p, 1\le j\le r},\qquad
D_{ij}=\begin{cases}
d_i & \text{if } i=j\le \rho,\\
0 & \text{otherwise},
\end{cases}
\quad d_i\ge 1.
\]
Put $\theta = Q^{-1}\phi$. Since $P$ is unimodular, the condition $R\phi \in (2\pi\mathbb{Z})^{p}$
is equivalent to $D\theta \in (2\pi\mathbb{Z})^{p}$, i.e. $d_j\theta_j\in 2\pi\mathbb{Z}$ for
$j=1,\dots,\rho$. Hence, using the $\theta$–coordinates,
\[
H=\{e^{i\theta_1}\in S^1:e^{i d_1\theta_1}=1\}\times \ldots \times \{e^{i\theta_\rho}\in S^1:e^{i d_\rho\theta_\rho}=1\} \times T^{r-\rho},
\]
so $\codim H=\rho$. Since $\codim H=s$, we conclude $\rho=s$.
For $j=1,\dots,s$ set $k_j := d_j (Q^{-1})^{T} e_j \in \mathbb{Z}^r$, where $e_j$ is the $j$-th standard basis vector of $\mathbb{R}^r$,
and note that $d_j\theta_j\in 2\pi\mathbb{Z}$ iff $\langle k_j,\phi\rangle\in 2\pi\mathbb{Z}$.
Hence $H=H_{k_1}\cap\ldots\cap H_{k_s}$.
\end{proof}

We now turn to intersections of the form $(H\times T^{l})\cap H_{(m,n)}$, which will play a central role in the bifurcation index calculations. We begin by recording their dimension.

\begin{Lemma}\label{lem:dim}
Let $H \in \sub(T^r)$ and let $(m,n)\in\mathbb Z^r\times\mathbb Z^l$. If $n\neq 0$, then
\[
\dim\big((H\times T^{l})\cap H_{(m,n)}\big)=l+\dim H-1.
\]
\end{Lemma}
\begin{proof}
By Lemma~\ref{lem:codimH} there exist linearly independent $k_1,\dots,k_s\in\mathbb{Z}^r$ with
$H=\bigcap_{j=1}^s H_{k_j}$, where $s:=\codim H$. Define
\[
F\colon \mathbb{R}^r\times\mathbb{R}^l\to \mathbb{C}^{s+1},\quad
F(\phi,\psi):=\bigl(e^{i\langle k_1,\phi\rangle}-1,\dots,e^{i\langle k_s,\phi\rangle}-1,\ e^{i(\langle m,\phi\rangle+\langle n,\psi\rangle)}-1\bigr).
\]
Then $F^{-1}(0)=(H\times T^{l})\cap H_{(m,n)}$. At any $(\phi,\psi)\in F^{-1}(0)$, the real Jacobian $F'(\phi,\psi)$ has the same rank as the matrix with rows
$
(k_1,0), \dots, (k_s,0), (m,n).
$
This rank is $s+1$: the first $s$ rows are linearly independent, while the last row has a nonzero $\psi$--part because $n\neq 0$, and hence is independent of them. The implicit function theorem yields
\[
\dim F^{-1}(0)=(r+l)-(s+1)=l+(r-s)-1=l+\dim H-1.
\]
\end{proof}

Using the above lemma with \eqref{eq:UTmultiplication}, we immediately obtain:
\begin{Corollary}\label{cor:nontr-prod}
If $n\neq 0$, then
$\chi_{T^{r+l}} ((T^r/H\times T^{ l})^+ )\star \chi_{T^{r+l}} ((T^{r}/H_{(m,n)})^+ )\neq \Theta$.
\end{Corollary}

We are now going to show that a sum of elements as in the above corollary is never trivial, provided $n\neq 0$. To that end we will need the following:

\begin{Lemma}\label{lem:intersections}
Let $r,l\ge 1$, let $H,H'\in\sub(T^r)$ and $m,m'\in\mathbb{Z}^r$, $n,n'\in\mathbb{Z}^l$.
\begin{enumerate}
\item If exactly one of $n,n'$ is zero, then
$
(H\times T^l)\cap H_{(m,n)}\ne (H'\times T^l)\cap H_{(m',n')}.
$
\item If $n\ne 0$ and $n'\ne 0$ and
$
(H\times T^l)\cap H_{(m,n)}=(H'\times T^l)\cap H_{(m',n')},
$
then $H=H'$.
\end{enumerate}
\end{Lemma}

\begin{proof}
(1) Assume without loss of generality that $n=0$ and $n'\ne 0$. Then
\[
(H\times T^l)\cap H_{(m,0)}=(H\cap H_m)\times T^l,
\]
so in particular $\{e\}\times T^l\subset (H\times T^l)\cap H_{(m,0)}$. On the other hand,
\[
(H'\times T^l)\cap H_{(m',n')}\cap(\{e\}\times T^l)=\{e\}\times H_{(0,n')},
\]
which is a proper subset of $\{e\}\times T^l$ because $n'\ne 0$. Hence the intersections cannot be equal.

(2) Assume $n\ne 0$ and $n'\ne 0$ and suppose
$$
(H\times T^l)\cap H_{(m,n)}=(H'\times T^l)\cap H_{(m',n')}.
$$
Let $\pi_1:T^r\times T^l\to T^r$ be the projection onto the first factor. For $n\ne 0$ one has
$
\pi_1\big((H\times T^l)\cap H_{(m,n)}\big)=H.
$
Indeed,
$$(H\times T^l)\cap H_{(m,n)}\subseteq H\times T^l\Rightarrow
\pi_1\big((H\times T^l)\cap H_{(m,n)}\big)\subseteq \pi_1(H\times T^l) =H.$$
Conversely, for every $e^{i\phi}\in H$ the surjectivity of $T^l\to S^1$, $e^{i\psi}\mapsto e^{i\langle n,\psi\rangle}$ (since $n\ne 0$), provides $e^{i\psi}\in T^l$ with $e^{i\langle n,\psi\rangle}=e^{-i\langle m,\phi\rangle}$, hence $(e^{i\phi},e^{i\psi})\in(H\times T^l)\cap H_{(m,n)}$, and applying $\pi_1$ gives $e^{i\phi}\in\pi_1((H\times T^l)\cap H_{(m,n)})$. Since $e^{i\phi}\in H$ was arbitrary, this shows $H\subseteq\pi_1((H\times T^l)\cap H_{(m,n)})$.
Similarly,
$
\pi_1\big((H'\times T^l)\cap H_{(m',n')}\big)=H'.
$
Taking $\pi_1$ of the assumed equality gives $H=H'$, as claimed.
\end{proof}

The next lemma will be used to establish nontriviality of the bifurcation index in the next section.

\begin{Lemma}\label{lem:inters}
Consider the elements of $U(T^{r+l})$ defined by
\[
A:=\sum_{H\in\sub(T^{r})\setminus\{T^{r}\}} n_{H}\,\chi_{T^{r+l}}\bigl(T^{r+l}/(H\times T^{l})^{+}\bigr),\quad
B:=\sum_{(m,n)\in\mathbb{Z}^{r+l}\setminus\{0\}} q_{(m,n)}\,\chi_{T^{r+l}}\bigl(T^{r+l}/H_{(m,n)}^{+}\bigr).
\]
Assume all nonzero $q_{(m,n)}$ have the same sign and that $q_{(m_0,n_0)}\ne 0$ for some $(m_{0},n_{0})$ with $n_{0}\ne 0$. If $A\neq\Theta$, then $A\star B\ne \Theta$.
\end{Lemma}

\begin{proof}
Fix $(m_0,n_0)$ with $n_0\ne 0$ and $q_{(m_0,n_0)}\ne 0$. Choose $H$ with $n_H\ne 0$ and set $$K_H:=(H\times T^l)\cap H_{(m_0,n_0)}\subset T^{r+l}.$$ The coefficient of $\chi_{T^{r+l}}(T^{r+l}/K_H^+)$ in $A\star B$ equals the sum of $n_{H'}q_{(m,n)}$ over all $H'\in\sub(T^r)\setminus\{T^r\}$ and $(m,n)\in\mathbb{Z}^{r+l}\setminus\{0\}$ with $(H'\times T^l)\cap H_{(m,n)}=K_H$. By Lemma~\ref{lem:intersections}, this forces  $H'=H$, so the coefficient is $n_H\sum q_{(m,n)}$ taken over all $(m,n)$ with $(H\times T^l)\cap H_{(m,n)}=K_H$. This sum contains $q_{(m_0,n_0)}$ and all nonzero $q_{(m,n)}$ have the same sign, hence the sum is nonzero; since $n_H\ne 0$, the coefficient is nonzero. Therefore $A\star B\ne \Theta$.
\end{proof}

\subsection{The degree}\label{subsec:degree}

Let $\bV$ be a finite-dimensional orthogonal $T^r$-representation, let $\varphi\in C^{1}(\bV,\mathbb{R})$ be $T^r$-invariant, and let $\Omega\subset\bV$ be open, bounded, $T^r$-invariant with $\partial\Omega\cap(\nabla\varphi)^{-1}(0)=\emptyset$. Gęba has defined in \cite{Geba}, for any compact Lie group $G$, a degree $\deg^{\nabla}_{G}(\nabla\varphi,\Omega)$ for $G$-equivariant gradient maps. Rybicki has extended this to the infinite-dimensional setting (see \cite{Ryb2005milano}): if $\bH$ is a separable Hilbert space being an orthogonal $G$-representation and $\Phi\in C^{1}(\bH,\mathbb{R})$ is $G$-invariant with $\nabla\Phi(u)=u-\nabla\eta(u)$, where $\nabla\eta\colon\bH\to\bH$ is $G$-equivariant and completely continuous, then for any open, bounded, $G$-invariant $\Omega\subset\bH$ with $\partial\Omega\cap(\nabla\Phi)^{-1}(0)=\emptyset$ one obtains a degree $\deg^{\nabla}_{G}(\nabla\Phi,\Omega)$. In this paper we specialize to $G=T^r$ and keep the notation $\deg^{\nabla}_{T^r}$ in both the finite- and infinite-dimensional cases; context will indicate which one is meant.

We record the basic computational tool for $\deg^{\nabla}_{T^r}(-\mathrm{Id},B(\bV))$.
Write $B(\bV)$ for the unit ball of $\bV$ and set $S^{\bV}:=B(\bV)/\partial B(\bV)$.
By \cite{Geba},
$\deg^{\nabla}_{T^r} (-\mathrm{Id},B(\bV) )=\chi_{T^r}(S^{\bV}).$
Combining this identification with \cite[Thm.~3.2]{GarRyb} yields the following theorem.

\begin{Theorem}\label{thm:UTnmult}
Let $\bV=\bR[k_0,0]\oplus\bigoplus_{i=1}^{s}\bR[k_i,m_i]$ be a $T^r$-representation. Then
\[
\deg^{\nabla}_{T^r}\bigl(-\mathrm{Id},B(\bV)\bigr)=(-1)^{k_0} \left(\bI-\sum_{i=1}^{s}k_i \chi_{T^r}(T^r/H_{m_i}^{+})\right)+\cO(\bV),
\]
where $\cO(\bV)$ is a $\mathbb{Z}$-linear combination of terms $\chi_{T^r}(T^r/H^{+})$ with $\dim H\le \dim T^r-2$.
\end{Theorem}

Since $\bR[k_i,m_i]$ are even dimensional, $(-1)^{k_0}=(-1)^{\dim \bV}$, hence, for $\bV=\bR[k_0,0]\oplus\bigoplus_{i=1}^{s}\bR[k_i,m_i]$,
\begin{equation}\label{eq:deg-id}
\deg^{\nabla}_{T^r}\bigl(-\mathrm{Id},B(\bV)\bigr)=(-1)^{\dim\bV} \left(\bI-\sum_{i=1}^{s}k_i \chi_{T^r}(T^r/H_{m_i}^{+})\right)+\cO(\bV).
\end{equation}

\begin{Remark}\label{rem:lift-degree-to-big-torus}
Let $\bV$ be a $T^{r}$-representation, $\Omega\subset \bV$ a $T^{r}$-invariant set, and $\varphi\in C^{1}(\bV,\mathbb{R})$ a $T^{r}$-invariant function with $\partial\Omega\cap(\nabla\varphi)^{-1}(0)=\emptyset$. Viewing $\bV$ as a $T^{r+l}=T^{r}\times T^{l}$-representation by letting $T^{l}$ act trivially, by the definition of the degree, if
\[
\deg^{\nabla}_{T^{r}}(\nabla\varphi,\Omega)
= a_{0} \chi_{T^{r}}(T^{r}/T^{r})+\sum_{H\in\sub(T^{r})} a_{H} \chi_{T^{r}}(T^{r}/H),
\]
 then
\[
\deg^{\nabla}_{T^{r+l}}(\nabla\varphi,\Omega)
= a_{0} \chi_{T^{r+l}}(T^{r+l}/T^{r+l})+\sum_{H\in\sub(T^{r})} a_{H} \chi_{T^{r+l}}\bigl(T^{r+l}/(H\times T^{l})\bigr).
\]
\end{Remark}

\section{Main result}\label{sec:main}

In this section we state and prove the main results. We begin by introducing the system and the standing assumptions, then describe the variational structure and formulate the bifurcation problem under consideration. Next we establish global bifurcation and symmetry--breaking theorems. We conclude with the case of a symmetric space, where we prove a stronger structural statement on the unboundedness of the bifurcating sets of solutions.

\subsection{Setting and assumptions}\label{subsec:system}

We study bifurcations of weak solutions to
\begin{equation}\label{eq:GH}
-\Delta_{M}u=\nabla_{u}F(u,\lambda) \quad\text{in } M\setminus\partial M,
\end{equation}
where $\lambda\in\mathbb{R}$, $M$ is a compact smooth Riemannian manifold and $-\Delta_{M}$ is the Laplace--Beltrami operator. The unknown is a map $u:M\to\bR^p$; thus $M$ is the domain of the equation, while $\bR^p$ is the codomain of the solution vector. When $\partial M\neq\emptyset$ we impose the Neumann boundary condition
\begin{equation}\label{eq:GH-Neumann}
\partial_{\nu}u=0 \quad\text{on } \partial M,
\end{equation}
where $\nu$ is the outward unit normal vector. Thus \eqref{eq:GH}--\eqref{eq:GH-Neumann} covers both cases, the boundary condition is void if $\partial M=\emptyset$.

We work under the following assumptions.

\begin{enumerate}
\item[(B1)] $M$ is a compact smooth Riemannian manifold with a smooth isometric action of an infinite compact Lie group $G$, so that $-\Delta_{M}$ is $G$--equivariant.

\item[(B2)] $\bR^{p}$ is an orthogonal representation of a torus $T^{r}$.

\item[(B3)] $F\in C^{2}(\bR^{p}\times\bR,\bR)$ is $T^{r}$-invariant in the first variable, i.e. $F(\tau u,\lambda)=F(u,\lambda)$ for all $\tau\in T^{r}$, $u\in\bR^{p}$, $\lambda\in\bR$. There exist $C>0$ and
\[
s\in
\begin{cases}
[1,(\dim M+2)/(\dim M-2))&\text{if }\dim M\ge 3,\\
[1,+\infty)&\text{if }\dim M\le 2,
\end{cases}
\]
such that
$
|\nabla^{2}_{u}F(u,\lambda)|\le C (1+|u|^{s-1})$ for all $(u,\lambda)\in\bR^{p}\times\bR.
$

\item[(B4)] There exists $u_{0}\in\bR^{p}$ with $\nabla_{u}F(u_{0},\lambda)=0$ for all $\lambda\in\bR$. After the change of variables $v=u-u_{0}$ we assume $u_{0}=0$.

\item[(B5)] There exists a symmetric matrix $A$ such that
$\nabla^{2}_{u}F(0,\lambda)=\lambda A$ for all $\lambda\in\bR.$

\item[(B6)] There exists $\delta>0$ such that $B_{\delta}(\bR^{p})\cap(\nabla_{u}F(\cdot,\lambda))^{-1}(0)=\{0\}$ and
\[
\deg^{\nabla}_{T^{r}}\bigl(\nabla_{u}F(\cdot,\lambda),B_{\delta}(\bR^{p})\bigr)\ne \Theta\in U(T^{r})
\]
for all $\lambda\in\bR\setminus\{0\}$. Here $B_{\delta}(X)$ denotes the open ball of radius $\delta$ in a normed space $X$.
\end{enumerate}

\begin{Remark}\label{rem:B6}
Assumption \textnormal{(B6)} may appear technical, but it is easy to verify in standard situations. Following the argument of \cite[Remark~4.1]{GolRybSte2025}, we note:
\begin{enumerate}[(i)]
\item If $0$ is a nondegenerate critical point of $F(\cdot,\lambda)$, then for $\delta>0$ small the degree
$\deg^{\nabla}_{T^{r}} (\nabla_{u}F(\cdot,\lambda),B_{\delta}(\mathbb{R}^{p}) )$
is nontrivial, hence \textnormal{(B6)} holds.
\item More generally, if the Brouwer degree $\deg_{B} (\nabla_{u}F(\cdot,\lambda),B_{\delta}(\mathbb{R}^{p}),0 )$ is nonzero (for example, when $0$ is a strict local minimum or maximum), then the $T^{r}$–equivariant degree is nontrivial, so \textnormal{(B6)} holds.
\end{enumerate}
As noted in \cite[Remark~4.3]{GolRybSte2025} (treated there for $r=1$; the same reasoning extends to any $r$), condition \textnormal{(B6)} is not limited to these cases.

Moreover, when the original symmetry is $S^{1}$, we may enlarge the symmetry group by embedding $S^{1}\subset T^{r}$ and letting the complementary torus act trivially. As in Remark~\ref{rem:lift-degree-to-big-torus}, nontriviality of the $S^{1}$–equivariant degree then lifts to nontriviality of the $T^{r}$–equivariant degree. In particular, \textnormal{(B6)} covers the $S^{1}$ condition \textnormal{(a3.2)} from \cite{GolRybSte2025}; the converse generally fails, so \textnormal{(B6)} is strictly more general.
\end{Remark}

\begin{Remark}
The potential $F$ may be symmetric under any compact Lie group $\Gamma$, not necessarily a torus. One could try to replace \textnormal{(B6)} by nontriviality of the $\Gamma$–equivariant degree. However, nontriviality does not, in general, pass to a maximal torus under restriction and we would then have to work in the Euler ring of a product with $\Gamma$, whose multiplicative structure is generally not explicit.
For this reason we impose \textnormal{(B6)} at the torus level requiring the $T^{r}$–equivariant degree to be nontrivial. This guarantees that later computations can be carried out in the Euler ring of a torus, where the structure is explicit.
\end{Remark}

We study weak solutions of \eqref{eq:GH}--\eqref{eq:GH-Neumann} as critical points of an associated functional on a Sobolev space.
Let $M$ be as in (B1). Denote by $H^{1}(M)$ the Sobolev space with inner product
\[
\langle v,w\rangle_{H^{1}(M)}:=\int_{M} (\nabla v(x),\nabla w(x))\,dM(x)+\int_{M} v(x)\cdot w(x)\,dM(x).
\]
Here $(\cdot,\cdot)$ is the Riemannian inner product and $\cdot$ is the Euclidean product on $\bR^p$. Set
$\bH:=\bigoplus_{i=1}^{p}H^{1}(M)$, with
$\langle u,v\rangle_{\bH}:=\sum_{i=1}^{p}\langle u_{i},v_{i}\rangle_{H^{1}(M)}$.

Weak solutions of \eqref{eq:GH}--\eqref{eq:GH-Neumann} are in one-to-one correspondence with critical points of
\[
\Phi(u,\lambda):=\tfrac12\int_{M}|\nabla u(x)|^{2}\,dM(x)-\int_{M}F(u(x),\lambda)\,dM(x),
\qquad (u,\lambda)\in\bH\times\bR.
\]
By the standard argument (see \cite{Rabinowitz1986}), $\nabla_u\Phi$ is a completely continuous perturbation of the identity. A routine computation gives
$\nabla_u\Phi(u,\lambda)=u-L_{\lambda A}u+\nabla_u\eta(u,\lambda),$
where $L_{\lambda A}\colon\bH\to\bH$ is defined by
\[
\langle L_{\lambda A}u,v\rangle_{\bH}=\int_{M}\bigl(u(x)+\lambda A u(x)\bigr)\cdot v(x)\,dM(x),\qquad u,v\in\bH,
\]
and $\nabla_u\eta(\cdot,\lambda)\colon\bH\to\bH$ is compact with $\nabla_u\eta(0,\lambda)=0$ and $\nabla_u^2\eta(0,\lambda)=0$ for all $\lambda\in\bR$. The operator $L_{\lambda A}$ is bounded, self-adjoint, and completely continuous on $\bH$.

By (B1)--(B3), $\bH$ is an orthogonal $T^{r}\times G$–representation with action $((\tau,g)\cdot u)(x):=\tau u(g^{-1}x)$ for $\tau\in T^{r}$, $g\in G$, $u\in\bH$, $x\in M$. With this action the functional $\Phi$ is invariant and the gradient is equivariant:
\[
\nabla_{u}\Phi((\tau,g)\cdot u,\lambda)=(\tau,g)\cdot\nabla_{u}\Phi(u,\lambda)\quad\text{for }\tau\in T^{r},\ g\in G,\ \lambda\in\bR.
\]

Let $\sigma(-\Delta_{M}):=\{0=\beta_{1}<\beta_{2}<\cdots\}$ be the spectrum of the Laplace--Beltrami operator on $M$ (with Neumann boundary condition if $\partial M\ne\emptyset$), and write $\bV_{-\Delta_M}(\beta_k)$ for the corresponding eigenspace. Let $\sigma(A)$ be the spectrum of $A$. Arguing as in \cite[Lemma~3.2]{GolKlu},
\begin{equation}\label{eq:spectrum}
\sigma\bigl(\mathrm{Id}-L_{\lambda A}\bigr)
=\left\{\frac{\beta_{k}-\lambda\alpha_{j}}{1+\beta_{k}}:\ \alpha_{j}\in\sigma(A),\ \beta_{k}\in\sigma(-\Delta_{M})\right\}.
\end{equation}
Let $\cV_0(\lambda) := \ker(\mathrm{Id}-L_{\lambda A})$,
$\bH_{0}:=\bigoplus_{i=1}^{p}\bV_{-\Delta_{M}}(0)$ and set $\cV(\lambda):=\cV_{0}(\lambda)\cap\bH_0^{\perp}$. By \eqref{eq:spectrum} and Lemma~\ref{lem:external-tensor},
\begin{equation}\label{eq:cvlambda}
\cV(\lambda)\approx_{T^r\times G}
\bigoplus_{\alpha_{j}\in\sigma(A)}\ \bigoplus_{\substack{\beta_{k}\in\sigma(-\Delta_{M})\setminus\{0\}\\ \beta_{k}=\lambda \alpha_{j}}}
\bigl(\bV_{A}(\alpha_{j})\otimes \bV_{-\Delta_{M}}(\beta_{k})\bigr),
\end{equation}
where $\bV_{A}(\alpha_{j})$ is the eigenspace of $A$ with eigenvalue $\alpha_{j}$.

Set
\begin{equation}\label{eq:Lambda}
\Lambda:=\Bigl\{\frac{\beta_{k}}{\alpha_{j}}:\ \alpha_{j}\in\sigma(A)\setminus\{0\},\ \beta_{k}\in\sigma(-\Delta_{M}) \Bigr\}.
\end{equation}
Then, for $\lambda \neq 0$, $\cV(\lambda)\neq\{0\}$ if and only if $\lambda\in\Lambda$. The set $\sigma(A)\setminus\{0\}$ is finite. Moreover, since $M$ is compact, the spectrum of $-\Delta_M$ is discrete, each eigenspace is finite-dimensional and $\beta_k\to+\infty$. Hence $\Lambda$ has no accumulation points. In particular, $\Lambda$ is discrete, i.e. every $\lambda_0\in\Lambda$ is isolated.

Let $\cW_{0}(\lambda)$ be the direct sum of the eigenspaces of $\operatorname{Id}-L_{\lambda A}$ corresponding to negative eigenvalues and set $\cW(\lambda):=\cW_{0}(\lambda)\cap\bH_0^{\bot}$. By \eqref{eq:spectrum} and Lemma~\ref{lem:external-tensor},
\begin{equation*}\label{eq:cWlambda}
\cW(\lambda)\approx_{T^r\times G}
\bigoplus_{\alpha_{j}\in\sigma(A)}\ \bigoplus_{\substack{\beta_{k}\in\sigma(-\Delta_{M})\setminus\{0\}\ \beta_{k}<\lambda \alpha_{j}}}
\bigl(\bV_{A}(\alpha_{j})\otimes \bV_{-\Delta_{M}}(\beta_{k})\bigr).
\end{equation*}
For fixed $\lambda$, only finitely many $\beta_k$ can satisfy $\beta_k<\lambda\alpha_j$ for some $\alpha_j\in\sigma(A)$. Hence, by the finite-dimensionality of the corresponding eigenspaces, $\cW(\lambda)$ is finite-dimensional.
Moreover, for $\lambda_0\in\bR$,
\begin{equation}\label{eq:cW-piecewise}
\cW(\lambda_{0})=
\begin{cases}
\displaystyle \bigoplus\limits_{\lambda\in \Lambda\cap(0,\lambda_{0})} \cV(\lambda), & \lambda_{0}>0,\\
\{0\}, & \lambda_{0}=0,\\
\displaystyle \bigoplus\limits_{\lambda\in \Lambda\cap(\lambda_{0},0)} \cV(\lambda), & \lambda_{0}<0.
\end{cases}
\end{equation}

\subsection{Global bifurcations}\label{subsec:global}
Consider the weak-solution equation for \eqref{eq:GH}--\eqref{eq:GH-Neumann} \begin{equation}\label{eq:bifeq}
\nabla_u \Phi(u,\lambda)=0.
\end{equation}
By (B4), the zero function is a solution for every $\lambda\in\bR$. Set
$\cT:=\{0\}\times\bR\subset\bH\times\bR$
and
\[\cN:=\{(u,\lambda)\in\bH\times\bR: \nabla_u\Phi(u,\lambda)=0,\ u\neq 0\}.
\]
We call $\cT$ the set of trivial solutions and $\cN$ the set of nontrivial solutions. Let $\cC(\lambda_0)$ denote the connected component of $\cT\cup\cN$ that contains $(0,\lambda_0)\in\cT$.

A point $(0,\lambda_{0})$ is a global bifurcation point if $\cC(\lambda_{0})\neq\{(0,\lambda_{0})\}$ and either $\cC(\lambda_{0})$ is unbounded in $\bH\times\bR$ or $\cC(\lambda_{0})\cap\cT$ contains a point distinct from $(0,\lambda_{0})$.
A local bifurcation from $(0,\lambda_{0})$ occurs if $(0,\lambda_{0})$ is an accumulation point of nontrivial solutions of \eqref{eq:bifeq}.

For a $G$–representation $\bW$ denote its fixed-point subspace by $\bW^{G}:=\{w\in\bW:\ g\cdot w=w\ \text{for all }g\in G\}$.

\begin{Lemma}\label{lem:nec}
Assume \textnormal{(B1)}--\textnormal{(B6)} and let $\lambda_{0}\in\bR\setminus\{0\}$ be such that $(0,\lambda_0)\in\mathcal T$ is a local bifurcation point of weak solutions of \eqref{eq:GH} and \eqref{eq:GH-Neumann}. If at least one of the following holds:
\begin{enumerate}
\item[(N1)] $\det A\neq 0$, i.e. $0$ is a nondegenerate critical point of $F(\cdot,\lambda)$ for all $\lambda\neq 0$,
\item[(N2)] $\bV_{-\Delta_M}(\beta_k)^{G}=\{0\}$ for every $\beta_k\neq 0$,
\end{enumerate}
then $\displaystyle \lambda_{0}=\frac{\beta_{k}}{\alpha_{j}}$ for some $\beta_{k}\in\sigma(-\Delta_{M})\setminus\{0\}$ and some $\alpha_{j}\in\sigma(A)\setminus\{0\}$.
\end{Lemma}

In case \textnormal{(N1)} the conclusion follows from the implicit function theorem together with \eqref{eq:spectrum}. In case \textnormal{(N2)} the same necessary condition is obtained by the argument of \cite[Lemma~3.5]{GolKluSte2}.

By the lemma, every bifurcation parameter $\lambda_{0}\neq 0$ lies in $\Lambda$, defined in \eqref{eq:Lambda}, so under the above assumptions $\Lambda$ is the set of candidate local (and hence potential global) bifurcation values. The case $\lambda_0=0$ is not covered by the lemma; we therefore include it among the candidates (it can arise only when $\beta_k=0$).

Let us now turn the attention to the necessary condition for global bifurcation.
Let $T^{l}$ be a maximal torus of $G$. Since $G$ is an infinite compact Lie group, we have $l>0$, see \cite[Chapter~IV]{BrockerDieck}.
Fix the torus
$\bT:=T^{r}\times T^{l}$ and $\lambda_{0}\in\Lambda$. Since $\Lambda$ is discrete, there exists $\varepsilon>0$ with $[\lambda_{0}-\varepsilon,\lambda_{0}+\varepsilon]\cap\Lambda=\{\lambda_{0}\}$. Hence there exists $\delta>0$ such that
$
\nabla_{u}\Phi(\cdot,\lambda_{0}\pm\varepsilon)^{-1}(0)\cap B_{\delta}(0,\bH)=\{0\},
$
so the degrees $\deg^{\nabla}_{\bT}(\nabla_{u}\Phi(\cdot,\lambda_{0}\pm\varepsilon),B_{\delta}(\bH))$ are well defined. Define the bifurcation index
\begin{equation}\label{eq:indbif}
\bif_{\bT}(\lambda_{0})
:=
\deg^{\nabla}_{\bT}\bigl(\nabla_{u}\Phi(\cdot,\lambda_{0}+\varepsilon),B_{\delta}(\bH)\bigr)
-
\deg^{\nabla}_{\bT}\bigl(\nabla_{u}\Phi(\cdot,\lambda_{0}-\varepsilon),B_{\delta}(\bH)\bigr)
\ \in\ U(\bT).
\end{equation}
The index $\bif_{\bT}(\lambda_{0})$ measures the change of the $\bT$--equivariant degree as $\lambda$ crosses $\lambda_{0}$. By the Rabinowitz-type global bifurcation theorem \cite[Theorem~4.9]{Ryb2005milano}, a nontrivial change forces global bifurcation from $(0,\lambda_{0})\in\mathcal T$. Thus we obtain:

\begin{Theorem}\label{thm:bif-index-global}
If $\bif_{\bT}(\lambda_{0})\neq\Theta\in U(\bT)$, then $(0,\lambda_0)\in\mathcal T$ is a global bifurcation point.
\end{Theorem}

We will use the bifurcation index and Theorem~\ref{thm:bif-index-global} to study global bifurcations of \eqref{eq:bifeq}.
The next result is one of the main theorems of this paper and provides a sufficient condition for global bifurcation at the values $\lambda_{0}\in\Lambda$.

\begin{Theorem}\label{thm:global-bifurcation}
Assume \textnormal{(B1)}--\textnormal{(B6)}. Fix $\beta_{k_{0}}\in\sigma(-\Delta_{M})\setminus\{0\}$ and $\alpha_{j_{0}}\in\sigma(A)\setminus\{0\}$, and set $\lambda_{0}:=\beta_{k_{0}}/\alpha_{j_{0}}$.
Suppose that $\bV_{-\Delta_{M}}(\beta_{k_{0}})$ is a nontrivial $T^{l}$-representation.
If, moreover, \textnormal{(N1)} or \textnormal{(N2)} holds,
then $(0,\lambda_{0})\in\cT$ is a global bifurcation point.
\end{Theorem}

\begin{proof}
For $\varepsilon>0$ as above, set $\Phi_{\pm}(u):=\Phi(u,\lambda_{0}\pm\varepsilon)$ and let
$\cL_{\pm}:=\nabla^{2}_{u}\Phi_{\pm}(0)=\mathrm{Id}-L_{(\lambda_{0}\pm\varepsilon)A}$.
By the choice of $\varepsilon$, the degrees $\deg^{\nabla}_{\bT}(\nabla_{u}\Phi_{\pm},B_{\delta}(\bH))$, where $\delta>0$ is sufficiently small, are well-defined.
By Theorem~\ref{thm:bif-index-global}, it suffices to show that the index $\bif_{\bT}(\lambda_{0})$ from \eqref{eq:indbif} is nonzero.

Arguing as in the proof of Theorem~3.1 in \cite{GolRybSte2025}, we obtain
\begin{equation}\label{eq:splitted}
\deg^{\nabla}_{\bT}\bigl(\nabla\Phi_{\pm},B_{\delta}(\bH)\bigr)
=
\deg^{\nabla}_{\bT}\bigl((\nabla\Phi_{\pm})|_{\bH_{0}}, B_{\delta}(\bH_{0})\bigr)
\star
\deg^{\nabla}_{\bT}\bigl(\cL_{\pm}\!\mid_{\bH_{0}^{\perp}},B_{\delta}(\bH_{0}^{\perp})\bigr),
\end{equation}
so the computation splits into the (possibly degenerate) block on $\bH_{0}$ and an isomorphism on its $\bT$–invariant complement.

To compute the first factor, note that on $\bH_{0}$  we have
$
(\nabla\Phi_{\pm})|_{\bH_{0}}(u)=\nabla_{u}F(u,\lambda_{0}\pm\varepsilon).
$
By homotopy invariance of the equivariant degree and (B6), since $\lambda_{0}\neq 0$, the degrees
$\deg^{\nabla}_{T^{r}}(\nabla_{u}F(\cdot,\lambda_{0}\pm\varepsilon), B_{\delta}(\mathbb{R}^{p}))$
do not depend on the chosen level. By (B6) there exist integers $n_{0}$ and $\{n_{H}\}_{H\in\sub(T^{r})\setminus\{T^{r}\}}$, not all zero and only finitely many nonzero, such that
\[
\deg^{\nabla}_{T^{r}}\bigl(\nabla_{u}F(\cdot,\lambda_{0}\pm\varepsilon), B_{\delta}(\mathbb{R}^{p})\bigr)
=
n_{0}\,\bI+
\sum_{H\in\sub(T^{r})\setminus\{T^{r}\}} n_{H}\,\chi_{T^{r}}\bigl((T^{r}/H)^{+}\bigr)\in U(T^{r}).
\]
By Remark~\ref{rem:lift-degree-to-big-torus}, this lifts to
\[
\deg^{\nabla}_{\bT}\bigl((\nabla\Phi_{\pm})|_{\bH_{0}}, B_{\delta}(\bH_{0})\bigr)
=
n_{0}\,\bI+
\sum_{H\in\sub(T^{r})\setminus\{T^{r}\}} n_{H}\,\chi_{\bT}\bigl((\bT/(H\times T^{l}))^{+}\bigr)\in U(\bT).
\]

On the other hand, by \cite[Thm.~4.7]{Ryb2005milano},
\[
\deg^{\nabla}_{\bT}\bigl(\cL_{\pm}\mid_{\bH_{0}^{\perp}},B_{\delta}(\cR_{2})\bigr)
=\deg^{\nabla}_{\bT}\bigl(-\mathrm{Id},B_{\delta}(\cW(\lambda_{0}\pm\varepsilon))\bigr).
\]

From now on we assume $\lambda_0>0$; the case $\lambda_0<0$ is analogous.
Since $\lambda_0>0$, by \eqref{eq:cW-piecewise} and \eqref{eq:splitted},
\begin{equation}\label{eq:bifind-final}
\begin{aligned}
\bif_{\bT}(\lambda_{0})=\ &
\deg^{\nabla}_{\bT}\bigl((\nabla\Phi_{\pm})|_{\bH_0},B_{\delta}(\bH_0)\bigr)\star
\deg^{\nabla}_{\bT}\bigl(-\mathrm{Id},B_{\delta}(\cW(\lambda_{0}-\varepsilon))\bigr)
\\
&\star\bigl(\deg^{\nabla}_{\bT}\bigl(-\mathrm{Id},B_{\delta}(\cV(\lambda_{0}))\bigr)-\bI\bigr)
\end{aligned}
\end{equation}
and since $\deg^{\nabla}_{\bT}(-\mathrm{Id},B_{\delta}(\cW(\lambda_{0}-\varepsilon)))$ is invertible in $U(\bT)$ (see, e.g., \cite[Thm.~3.11]{GeRy}, \cite[Thm.~2.1]{GolRyb1}), the nontriviality of $\bif_{\bT}(\lambda_{0})$ reduces to
\begin{equation}\label{eq:reducedbif}
\deg^{\nabla}_{\bT}\bigl((\nabla\Phi_{\pm})|_{\bH_0},B_{\delta}(\bH_0)\bigr)
\ \star\
\bigl(\deg^{\nabla}_{\bT}\bigl(-\mathrm{Id},B_{\delta}(\cV(\lambda_{0}))\bigr)-\bI\bigr)
\ne \Theta.
\end{equation}

Since $\bV_{-\Delta_{M}}(\beta_{k_{0}})$ is a nontrivial $T^{l}$–representation, the tensor product $\bV_{A}(\alpha_{j_{0}})\otimes\bV_{-\Delta_{M}}(\beta_{k_{0}})$ is a nontrivial $\bT$–representation and consequently so is $\cV(\lambda_{0})$. Moreover, there exists $\bR[1,(m,n)]\subset\cV(\lambda_{0})$ with $n\neq 0$. Applying Theorem~\ref{thm:UTnmult} to $\cV(\lambda_{0})$ yields
\[
\deg^{\nabla}_{\bT}\bigl(-\mathrm{Id},B_{\delta}(\cV(\lambda_{0}))\bigr)
=(-1)^{q_{0}}\Bigl(\bI-\sum_{(m,n)\in\mathbb{Z}^{r+l}\setminus\{0\}} q_{(m,n)}\,\chi_{\bT}(\bT/H_{(m,n)}^{+})\Bigr)
+\cO\bigl(\cV(\lambda_{0})\bigr),
\]
where $q_{(m,n)}\in\mathbb{Z}$ with at least one $q_{(m,n)}\neq 0$ for some $n\neq 0$, and where $\cO(\cV(\lambda_{0}))$ is a $\mathbb{Z}$–linear combination of $\chi_{\bT}(\bT/H^{+})$ with $\dim H\le \dim\bT-2$.

To finish the proof, assume first that $n_{0}\ne 0$. By Corollary~\ref{cor:bigcodim}, the codimension-1 part of the product in \eqref{eq:reducedbif} is
\[
n_{0}\sum_{(m,n)\in\mathbb{Z}^{r+l}\setminus\{0\}} q_{(m,n)}\,\chi_{\bT}(\bT/H_{(m,n)}^{+})
+\bigl((-1)^{q_{0}}-1\bigr)\sum_{  \codim H=1 } n_{H}\,\chi_{\bT}\bigl(\bT/(H\times T^{l})^{+}\bigr)\in U(\bT),
\]
where $n_{H}\in\mathbb{Z}$ (possibly all zero). Since there exists $(m,n)$ with $n\ne 0$ and $q_{(m,n)}\ne 0$, and codimension-1 subgroups coming from $T^{r}$ appear in the second summand as $H_{(k,0)}=H\times T^{l}$, this element is nonzero. Hence $\bif_{\bT}(\lambda_{0})\ne\Theta$.

When $n_{0}=0$,
\[
\deg^{\nabla}_{\bT}\bigl((\nabla\Phi_{\pm})|_{\bH_0},B_{\delta}(\bH_0)\bigr)
=\sum_{H\in\sub(T^{r})\setminus \{T^r\}} n_{H} \chi_{\bT}\bigl(\bT/(H\times T^{l})^{+}\bigr).
\]
Let $c_0$ be the smallest codimension of $H$ for which $n_H\neq 0$.

If $\dim \cV(\lambda_0)$ is odd, the coefficient of $\bI$ in
$\deg^{\nabla}_{\bT}(-\mathrm{Id},B_{\delta}(\cV(\lambda_{0})))-\bI$ equals $-2$. Hence, using Corollary~\ref{cor:bigcodim}, the codimension-$c_0$ part of the product in \eqref{eq:reducedbif} is
$$
-2\sum_{\codim H=c_0} n_{H} \chi_{\bT}\bigl(\bT/(H\times T^{l})^{+}\bigr).
$$
By the choice of $c_0$ this sum is nonzero, and consequently $\bif_{\bT}(\lambda_{0})\neq \Theta$.

Suppose now that $\dim \cV(\lambda_0)$ is even.
Applying Lemma~\ref{lem:inters} with
\[
A=\sum_{\codim H=c_0} n_{H} \chi_{\bT}\bigl(\bT/(H\times T^{l})^{+}\bigr),\qquad
B=\sum_{(m,n)\in\mathbb{Z}^{r+l}\setminus\{0\}} q_{(m,n)} \chi_{\bT}(\bT/H_{(m,n)}^{+}),
\]
and using $q_{(m,n)}\ne 0$ with $n\ne 0$ as above, we obtain that the product in \eqref{eq:reducedbif} is nontrivial. Hence $\bif_{\bT}(\lambda_{0})\ne \Theta$.
Using Theorem \ref{thm:bif-index-global} we finish the proof.
\end{proof}

\begin{Remark}
If \textnormal{(N2)} holds, then $\bV_{-\Delta_M}(\beta_k)^{G}=\{0\}$ for every $\beta_k\ne 0$. Hence $\bV_{-\Delta_M}(\beta_k)$ is a nontrivial $G$–representation and, if $G$ is connected, by \cite[Lemma~3.3]{GarRyb}, a nontrivial $T^{l}$–representation. Thus, under \textnormal{(N2)} and for connected $G$, the assumption of Theorem~\ref{thm:global-bifurcation} is satisfied for all $\beta_{k}\in\sigma(-\Delta_M)\setminus\{0\}$.
\end{Remark}

In the absence of the nontrivial $T^{l}$–representation assumption on $\bV_{-\Delta_M}(\beta_{k_0})$, a parity condition on  the dimension of  $\cV(\lambda_0)$ suffices to force a nonzero bifurcation index. In particular, if $\dim\cV(\lambda_0)$ is odd, then global bifurcation occurs. Namely, we have the following:

\begin{Theorem}\label{thm:global-bifurcation-trivial}
Assume \textnormal{(B1)}--\textnormal{(B6)}. Fix $\beta_{k_{0}}\in\sigma(-\Delta_{M})\setminus\{0\}$ and $\alpha_{j_{0}}\in\sigma(A)\setminus\{0\}$, and set $\lambda_{0}:=\beta_{k_{0}}/\alpha_{j_{0}}$.
Assume that $\dim \cV(\lambda_0)$ is odd.
If, moreover, \textnormal{(N1)} or \textnormal{(N2)} holds,
then $(0,\lambda_{0})\in\cT$ is a global bifurcation point.
\end{Theorem}

\begin{proof}
 As in the proof of Theorem~\ref{thm:global-bifurcation}, it suffices to show
\begin{equation}\label{eq:reducedbif-trivial}
\deg^{\nabla}_{\bT}\bigl((\nabla\Phi_{\pm})|_{\bH_0},B_{\delta}(\bH_0)\bigr)
\ \star\
\bigl(\deg^{\nabla}_{\bT}\bigl(-\mathrm{Id},B_{\delta}(\cV(\lambda_{0}))\bigr)-\bI\bigr)
\ne \Theta,
\end{equation}
with the same notation as before.

By (B5) and Remark \ref{rem:lift-degree-to-big-torus} there exist integers $n_0$ and $\{n_H\}_{H\in\sub(T^r)\setminus\{T^r\}}$, not all zero, only finitely many nonzero, such that
\[
\deg^{\nabla}_{\bT}\bigl((\nabla\Phi_{\pm})|_{\bH_0},B_{\delta}(\bH_0)\bigr)
=
n_0 \bI+\sum_{H\in\sub(T^r)\setminus\{T^r\}}n_H \chi_{\bT}\bigl(\bT/(H\times T^l)^+\bigr).
\]
Since $\dim\cV(\lambda_0)$ is odd, we have
\[
\deg^{\nabla}_{\bT}\bigl(-\mathrm{Id},B_{\delta}(\cV(\lambda_{0}))\bigr)-\bI
= -2 \bI+\text{(terms with codimension $\ge1$)}.
\]

If $n_0\neq 0$, then the codimension-$0$ part of the product in \eqref{eq:reducedbif-trivial} equals $(-2n_0)\bI\neq 0$, so \eqref{eq:reducedbif-trivial} holds.

If $n_0=0$, let $c_0$ be the minimal codimension for which some $n_H\neq 0$. The codimension-$c_0$ part of the product in \eqref{eq:reducedbif-trivial} is then
\[
-2\sum_{\codim H=c_0} n_H \chi_{\bT}\bigl(\bT/(H\times T^l)^+\bigr)\neq 0,
\]
again proving \eqref{eq:reducedbif-trivial}. In both cases $\bif_{\bT}(\lambda_0)\neq\Theta$ and Theorem~\ref{thm:bif-index-global} yields the claim.
\end{proof}

\begin{Remark}
If $\bV_{-\Delta_M}(\beta_{k_0})$ is a trivial $G$–representation and $\dim \cV(\lambda_0)$ is even, the above results can still be refined. This would require a careful treatment of the degree in \textnormal{(B6)} and a detailed analysis of intersections of subgroups of a torus. The refinement is feasible but technical and adds little insight, so we omit it.

If the symmetry group $G$ is finite, it contains no nontrivial torus. In this case we are effectively confined to the $T^{r}$–equivariant framework and computations in $U(T^{r})$. One can still detect global bifurcation as in the theorem above, but the simple separation via the coefficient of $\chi_{\bT}(\bT/H_{(m,n)}^+)$ appearing in one factor and not in the other for some $n\neq 0$ is no longer available. Consequently, one must examine the structure of torus subgroups more closely to enforce bifurcation, and the resulting criteria for global bifurcation become substantially more technical.
\end{Remark}

We conclude this section with the case where the criterion of Lemma~\ref{lem:nec} fails. In this situation there is no a priori “suspected” level, but one can still track the variation of the equivariant degree. If the degree cannot be defined at some $\lambda$ (because nontrivial solutions accumulate at the trivial branch), then a local bifurcation occurs; if the degree is defined on both sides of a value and can be computed, a change in the degree necessarily signals a global bifurcation somewhere in between.

Let us formalize this. Arguing as in \cite[Theorem.~3.8]{GolKluSte2} and Theorem \ref{thm:global-bifurcation} and \ref{thm:global-bifurcation-trivial}, we obtain the following alternative.

\begin{Theorem}\label{thm:local-or-global}
Assume \textnormal{(B1)}--\textnormal{(B6)} and fix $\lambda_0\in\Lambda$.
Let $\alpha_{j_0}\in\sigma(A)\setminus\{0\}$ and $\beta_{k_0}\in\sigma(-\Delta_M)\setminus\{0\}$ satisfy $\lambda_0=\beta_{k_0}/\alpha_{j_0}$.
Assume that $\bV_{-\Delta_M}(\beta_{k_0})$ is a nontrivial $T^{l}$--rep\-re\-sen\-ta\-tion or that $\dim\cV(\lambda_0)$ is odd.
For every $\varepsilon>0$ such that
$(\lambda_0-\varepsilon,\lambda_0+\varepsilon)\cap
\Lambda
=\{\lambda_0\},$
at least one of the following holds:
\begin{enumerate}[(i)]
\item a local bifurcation of solutions of \eqref{eq:bifeq} occurs from $(0,\lambda)\in\cT$ for every $\lambda\in(\lambda_0-\varepsilon,\lambda_0)$ or for every $\lambda\in(\lambda_0,\lambda_0+\varepsilon)$;
\item a global bifurcation of solutions of \eqref{eq:bifeq} occurs from $(0,\widehat\lambda)\in\cT$ for some $\widehat\lambda\in(\lambda_0-\varepsilon,\lambda_0+\varepsilon)$.
\end{enumerate}
\end{Theorem}

\begin{Remark}
For $\lambda_{0}\neq 0$, in the proofs of Theorems~\ref{thm:global-bifurcation} and~\ref{thm:global-bifurcation-trivial} we use that the degrees
$
\deg^{\nabla}_{T^{r}}(\nabla_{u}F(\cdot,\lambda_{0}\pm\varepsilon), B_{\delta}(\mathbb{R}^{p}))
$
are independent of the choice of level $\lambda_{0}\pm\varepsilon$. For $\lambda_{0}=0$ this invariance may fail. Following these proofs, it still follows that if
\[
\deg^{\nabla}_{T^{r}}\bigl(\nabla_{u}F(\cdot,\varepsilon), B_{\delta}(\mathbb{R}^{p})\bigr)
\neq
\deg^{\nabla}_{T^{r}}\bigl(\nabla_{u}F(\cdot,-\varepsilon), B_{\delta}(\mathbb{R}^{p})\bigr),
\]
then global bifurcation occurs from $(0,0)\in\cT$, provided the assumptions of Lemma~\ref{lem:nec} are satisfied.
If $0$ is a nondegenerate critical point of $F(\cdot,\lambda)$ for all $\lambda\neq 0$ (i.e., \textnormal{(N1)} holds), then, by the linearization property of the degree, the question reduces to whether the numbers (counting multiplicities) of negative and positive eigenvalues of $A$ have different parity, which is precisely to say that $p$ is odd.
Finally, if the necessary condition of Lemma~\ref{lem:nec} fails at $\lambda_0=0$, then bifurcation occurs as in Theorem~\ref{thm:local-or-global}: locally for $\lambda$ near $0$, or globally at some $\widehat\lambda\in(-\varepsilon,\varepsilon)$.
\end{Remark}

\subsection{Symmetry breaking}\label{subsec:symmetry-breaking}
We now discuss \textnormal{(N2)} and its implications for the symmetry of bifurcating solutions. We say that symmetry breaking occurs at a bifurcation point $(0,\lambda_{0})$ if, in some neighbourhoods $U\subset\bH$ of $0$ and $I\subset\bR$ of $\lambda_{0}$, every nontrivial solution $(u,\lambda)\in U\times I$ of \eqref{eq:bifeq} satisfies $G_{u}\neq G$, where $G_{u}:=\{g\in G:\ g\cdot u=u\}$ denotes the isotropy group of $u$.

\begin{Theorem}\label{thm:symm-selection}
Assume \textnormal{(B1)}--\textnormal{(B6)} and \textnormal{(N2)}, and let $\lambda_{0}\in\bR\setminus\{0\}$ be such that $(0,\lambda_{0})\in\cT$ is a bifurcation point of weak solutions of \eqref{eq:GH}--\eqref{eq:GH-Neumann}. Then symmetry breaking occurs at $(0,\lambda_{0})\in\cT$.
\end{Theorem}
\begin{proof}
Assume, to the contrary, that there exists a nontrivial solution $(u,\lambda)$ with $\lambda$ near $\lambda_{0}\ne 0$ and $G_{u}=G$. Then $u\in\bH^{G}\subset\bH_{0}$ by \textnormal{(N2)}, so $u$ is constant. By \textnormal{(B6)}, the only constant solution in a neighborhood of $0$ is $u=0$, a contradiction. Hence every nontrivial solution near $(0,\lambda_{0})$ has $G_{u}\ne G$, i.e., symmetry breaking occurs.
\end{proof}

Under \textnormal{(N2)}, when $\lambda_{0}=0$ any bifurcation - if it occurs - is confined to $\bH_{0}$, hence to constant solutions $u(x)\equiv u_{0}$ with $\nabla_{u}F(u_{0},\lambda)=0$. Thus no symmetry breaking can occur at $(0,0)$. Indeed, by standard equivariant Lyapunov–Schmidt arguments (cf. \cite{Dancer1979} and \cite[Lemma~3.11]{GKS2018}), any solution $(u,\lambda)$ near $(0,0)$ has the same isotropy group as some element of $\cV_{0}(0)$, and $\cV_{0}(0)\subset\bH_{0}$ consists of $G$–invariant functions.

\subsection{Elliptic systems on a symmetric space}\label{subsec:unbounded}

In this section we specialize \eqref{eq:GH} to compact Riemannian symmetric spaces.
Let $G$ be a compact, connected, semisimple Lie group  and let $H\subset G$ be
the fixed-point subgroup of an involutive automorphism of $G$. Then
$M:=G/H$ is a compact Riemannian symmetric space endowed with its canonical
$G$-invariant metric.

On such manifolds the eigenspaces of $-\Delta_{M}$ are well understood, which
allows us to verify the assumptions of Theorems \ref{thm:global-bifurcation}
and \ref{thm:symm-selection} and thereby locate global bifurcation points at
which symmetry breaking occurs. Moreover, using the algebraic description of
these eigenspaces as $G$-representations,
we refine the global bifurcation theorem to obtain unboundedness of the
bifurcating continua. To that end, let us recall the following $\bT$–equivariant
version of the Rabinowitz alternative (cf. \cite[Thm.~4.9]{Ryb2005milano}),
together with its summation formula:

\begin{Theorem}\label{thm:altRab_sum_T}
If $\bif_{\bT}(\lambda_0)\neq \Theta$ in $U(\bT)$, then a global bifurcation of
solutions to \eqref{eq:bifeq} occurs from $(0,\lambda_0)\in\cT$. Moreover, if the
continuum $\cC(\lambda_0)$ is bounded, then the set $\cC(\lambda_0)\cap \cT$
is finite and
\begin{equation}\label{eq:sub_bif_T}
\sum_{(0,\widehat{\lambda})\in \cC(\lambda_0)\cap \cT}
\bif_{\bT}(\widehat{\lambda})=\Theta \in U(\bT).
\end{equation}
\end{Theorem}

In particular, if the individual indices $\bif_{\bT}(\widehat{\lambda})$
are known, \eqref{eq:sub_bif_T} can be used to deduce that $\cC(\lambda_0)$
is unbounded whenever the sum in \eqref{eq:sub_bif_T} cannot vanish.

We now record the properties of the Laplace--Beltrami eigenspaces which will be used
to compute these indices. The precise torus-equivariant properties needed below
are collected in \cite[Section~2.2]{Stefaniak2025}.

\begin{Lemma}\label{lem:eigenspaces-GH}
Let $M=G/H$ be a compact Riemannian symmetric space. Then:
\begin{enumerate}
\item For each $\beta\in\sigma(-\Delta_{M})\setminus\{0\}$, the eigenspace
$\bV_{-\Delta_{M}}(\beta)$ is a finite-dimensional, nontrivial $G$-representation.
\item $\bV_{-\Delta_{M}}(0)^G=\bV_{-\Delta_{M}}(0)$ and
$\bV_{-\Delta_{M}}(\beta)^G=\{0\}$ for $\beta>0$.
\item For every $\beta_k\in\sigma(-\Delta_M)$ there exists $\nu_k\in\mathbb{Z}^l$
such that $\bR[1,\nu_k]\subset\bV_{-\Delta_M}(\beta_k)$ while
$\bR[1,\nu_k]\not\subset\bV_{-\Delta_M}(\beta_m)$ for any
$\beta_m\in\sigma(-\Delta_M)$ with $\beta_m<\beta_k$.
\end{enumerate}
\end{Lemma}

As a consequence of Lemma~\ref{lem:eigenspaces-GH} together with
Theorems~\ref{thm:global-bifurcation} and~\ref{thm:symm-selection} we obtain:

\begin{Theorem}\label{thm:global-bifurcation-symmetric}
Assume \textnormal{(B1)}--\textnormal{(B6)} and let $M=G/H$ be a symmetric space with $G$ connected. Fix $\beta_{k_{0}}\in\sigma(-\Delta_{M})\setminus\{0\}$ and $\alpha_{j_{0}}\in\sigma(A)\setminus\{0\}$, and set $\lambda_{0}:=\beta_{k_{0}}/\alpha_{j_{0}}\in\Lambda$.
Then $(0,\lambda_{0})\in\cT$ is a global bifurcation point at which symmetry breaking occurs.
\end{Theorem}

We next ask when the continua from Theorem~\ref{thm:global-bifurcation-symmetric} are unbounded. For this we assume the following property of the eigenspaces of $A$:
\begin{enumerate}
\item[(E)] For every $\alpha_j\in\sigma(A)$ there exists $\mu_j\in\mathbb{Z}^r$ such that
$\bR[1,\mu_j]\subset\bV_A(\alpha_j)$ and
$\bR[1,\mu_j]\not\subset\bV_A(\alpha_i)$ for any $\alpha_i\in\sigma(A)$ with $\alpha_i\ne\alpha_j$.
\end{enumerate}

\begin{Lemma}\label{lem:highest}
Let $M$ be a compact symmetric space. Suppose that the eigenvalues of $A$ satisfy \textnormal{(E)}.
Fix $\lambda_0=\beta_{k_0}/\alpha_{j_0}\in\Lambda\setminus\{0\}$. Then:
\begin{enumerate}
\item If $\mu_{j_0}\neq 0$, then
$
\bR[1,(\mu_{j_0},\nu_{k_0})]\oplus \bR[1,(\mu_{j_0},-\nu_{k_0})]\subset \cV(\lambda_0).
$
\item If $\mu_{j_0}=0$, then
$
\bR[1,(0,\nu_{k_0})]\subset \cV(\lambda_0).
$
\end{enumerate}
Moreover, if  $\lambda_0>0$, then for every $\lambda<\lambda_0$,
$
\bR[1,(\mu_{j_0},\pm\nu_{k_0})]\not\subset \cV(\lambda),$
while if $\lambda_0<0$, the same conclusion holds for every $\lambda>\lambda_0$.
\end{Lemma}

\begin{proof}
Fix $\lambda_0=\beta_{k_0}/\alpha_{j_0}\in\Lambda\setminus\{0\}$. In particular, $\beta_{k_0}\neq 0$.
By (E) and Lemma~\ref{lem:eigenspaces-GH}(3) there exist
$\mu_{j_0}\in\mathbb{Z}^r$ and $\nu_{k_0}\in\mathbb{Z}^l\setminus\{0\}$ with
$
\bR[1,\mu_{j_0}]\subset \bV_A(\alpha_{j_0})$ and
$\bR[1,\nu_{k_0}]\subset \bV_{-\Delta_M}(\beta_{k_0}),
$
such that $\bR[1,\mu_{j_0}]$ does not occur in $\bV_A(\alpha_j)$ for $j\ne j_0$ and
$\bR[1,\nu_{k_0}]$ does not occur in $\bV_{-\Delta_M}(\beta_k)$ for $\beta_k<\beta_{k_0}$.
Using the definition of $\cV(\lambda_0)$  and
Lemma~\ref{lem:T-product-decomp}, we obtain
\[
\bR[1,(\mu_{j_0},\nu_{k_0})]\oplus \bR[1,(\mu_{j_0},-\nu_{k_0})]\subset
\bV_A(\alpha_{j_0})\otimes \bV_{-\Delta_M}(\beta_{k_0})
\subset \cV(\lambda_0)
\]
if $\mu_{j_0}\neq 0$, while, if $\mu_{j_0}=0$, then
$
\bR[1,(0,\nu_{k_0})]\subset \cV(\lambda_0).
$

To finish the proof, suppose that $\lambda\in\Lambda$ and $\bR[1,(\mu_{j_0},\nu_{k_0})]\subset \cV(\lambda)$ or $\bR[1,(\mu_{j_0},-\nu_{k_0})]\subset \cV(\lambda)$. Then, by Lemma~\ref{lem:T-product-decomp}, $\bR[1,(\mu_{j_0},\pm\nu_{k_0})]$ must come from a tensor summand $\bV_A(\alpha_j)\otimes \bV_{-\Delta_M}(\beta_k)$ with $\beta_k=\lambda\alpha_j$ and with $\bR[1,\mu_{j_0}]\subset \bV_A(\alpha_j)$ and $\bR[1,\pm\nu_{k_0}]\subset \bV_{-\Delta_M}(\beta_k)$. Since $\bR[1,\nu_{k_0}]\approx_{T^l}\bR[1,-\nu_{k_0}]$, the latter is equivalent to $\bR[1,\nu_{k_0}]\subset \bV_{-\Delta_M}(\beta_k)$. Hence $\beta_k\ge \beta_{k_0}$ and, by (E), $j=j_0$. Therefore, if $\lambda_0>0$ (so $\alpha_{j_0}>0$), then $\lambda=\beta_k/\alpha_{j_0}\ge \beta_{k_0}/\alpha_{j_0}=\lambda_0$, and if $\lambda_0<0$, then $\lambda=\beta_k/\alpha_{j_0}\le \beta_{k_0}/\alpha_{j_0}=\lambda_0$. This proves the claim.
\end{proof}

Combining Lemma~\ref{lem:highest} with the $\bT$–equivariant Rabinowitz alternative and its summation formula (Theorem~\ref{thm:altRab_sum_T}) shows that the bifurcation indices along $\cT$ do not cancel, and hence the continua through the bifurcation points cannot be bounded. Let us make this precise in the following theorem.

\begin{Theorem}\label{thm:unbounded}
Let $M$ be a compact symmetric space. Assume \textnormal{(B1)}--\textnormal{(B6)} and \textnormal{(E)}. Suppose that $0$ is nondegenerate. Then for every $\lambda_{0}\in\Lambda\setminus\{0\}$ the point $(0,\lambda_{0})\in\cT$ is a global bifurcation point and the continuum $\mathcal{C}(\lambda_0)$ is unbounded in $\bH\times\mathbb{R}$. If, additionally, $p$ is odd, then the continuum $\mathcal{C}(0)$ is also unbounded in $\bH\times\mathbb{R}$.
\end{Theorem}

\begin{proof}
Throughout the proof we use the notation of Theorem~\ref{thm:global-bifurcation}.
Fix $\lambda_{0}=\beta_{k_0}/\alpha_{j_0}\neq 0$. Without loss of generality assume $\lambda_0>0$; the case $\lambda_0<0$ is completely analogous. By Theorem~\ref{thm:global-bifurcation-symmetric}, $(0,\lambda_0)\in\cT$ is a global bifurcation point. We will show that the continuum $\mathcal{C}(\lambda_0)$ is unbounded.

Assume, to the contrary, that $\mathcal{C}(\lambda_0)$ is bounded. Then, by Theorem~\ref{thm:altRab_sum_T}, the set $\mathcal{C}(\lambda_0)\cap \mathcal{T}$ is finite and
\begin{equation}\label{eq:sub_bif_T_reorg}
\sum_{(0,\widehat{\lambda})\in \mathcal{C}(\lambda_0)\cap \mathcal{T}}
\bif_{\bT}(\widehat{\lambda})=\Theta\in U(\bT).
\end{equation}
Without loss of generality, assume $\lambda_0=\max\{\widehat{\lambda}:(0,\widehat{\lambda})\in \mathcal{C}(\lambda_0)\cap \mathcal{T}\}$.

Fix $\mu_{j_0}\in\bZ^r$ and $\nu_{k_0}\in\bZ^l\setminus \{0\}$ as in (E) and Lemma~\ref{lem:eigenspaces-GH}(3).
By Lemma~\ref{lem:highest}, $\bR[1,(\mu_{j_0},\nu_{k_0})]\subset \cV(\lambda_0)$ and $\bR[1,(\mu_{j_0},\nu_{k_0})]\not\subset \cV(\lambda)$ for all $\lambda<\lambda_0$.
Set $H_{*}:=H_{(\mu_{j_0},\nu_{k_0})}$.
We are going to show that the coefficient of $\chi_{\bT}(\bT/H_{*}^{+})$ in \eqref{eq:sub_bif_T_reorg} is nonzero.

Recall that (see \eqref{eq:bifind-final})
\begin{equation*}
\begin{aligned}
\bif_{\bT}(\lambda_{0})=\ &
\deg^{\nabla}_{\bT}\bigl((\nabla\Phi_{\pm})|_{\bH_0},B_{\delta}(\bH_0)\bigr)\star
\deg^{\nabla}_{\bT}\bigl(-\mathrm{Id},B_{\delta}(\cW(\lambda_{0}-\varepsilon))\bigr)
\\
&\star\bigl(\deg^{\nabla}_{\bT}\bigl(-\mathrm{Id},B_{\delta}(\cV(\lambda_{0}))\bigr)-\bI\bigr).
\end{aligned}
\end{equation*}
By \eqref{eq:cW-piecewise}, since $\lambda_0>0$,
$\cW(\lambda_0-\varepsilon)=\bigoplus_{\lambda\in \Lambda\cap(0,\lambda_{0})} \cV(\lambda).$
Thus $\bR[1,(\mu_{j_0},\nu_{k_0})]\not\subset \cW(\lambda_0-\varepsilon)$. Therefore, by Theorem~\ref{thm:UTnmult}, the coefficient of $\chi_{\bT}(\bT/H_{*}^{+})$ in $\deg^{\nabla}_{\bT}\bigl(-\mathrm{Id},B_{\delta}(\cW(\lambda_{0}-\varepsilon))\bigr)$ is $0$. The same happens for that coefficient in $\deg^{\nabla}_{\bT}\bigl((\nabla\Phi_{\pm})|_{\bH_0},B_{\delta}(\bH_0)\bigr)$ by Remark~\ref{rem:lift-degree-to-big-torus}, since $\nu_{k_0}\neq 0$.

Hence, using \eqref{eq:deg-id} and \eqref{eq:cW-piecewise}, we obtain that the coefficient of $\chi_{\bT}(\bT/H_{*}^{+})$ in $\bif_{\bT}(\lambda_{0})$ is
\[
n_{0}(-1)^{\dim\cW(\lambda_{0}-\varepsilon)}(-1)^{\dim\cV(\lambda_{0})+1}m_{H_*}
= -n_{0}(-1)^{\dim\cW(\lambda_{0}+\varepsilon)}m_{H_*}\neq 0,
\]
that is, the $\bI$–coefficients of the first two factors multiplied by the coefficient of $\chi_{\bT}(\bT/H_{*}^{+})$ in the third, where $m_{H_*}\ge 1$ is the multiplicity of $\bR[1,(\mu_{j_0},\nu_{k_0})]$ in $\cV(\lambda_0)$.

On the other hand, since $\bR[1,(\mu_{j_0},\nu_{k_0})]$ does not occur in any $\cV(\widehat{\lambda})$ for $\widehat{\lambda}<\lambda_0$, no other bifurcation index $\bif_{\bT}(\widehat{\lambda})$ in \eqref{eq:sub_bif_T_reorg} contributes to $\chi_{\bT}(\bT/H_{*}^{+})$. Hence the sum in \eqref{eq:sub_bif_T_reorg} cannot vanish, a contradiction. Therefore $\mathcal{C}(\lambda_0)$ is unbounded.

For $\lambda_0=0$, if $p$ is odd, Theorem~\ref{thm:altRab_sum_T} implies that either $\mathcal{C}(0)$ is unbounded or else $\mathcal{C}(0)$ meets $(0,\lambda_1)$ for some $\lambda_1\in\Lambda\setminus\{0\}$. Since each $\mathcal{C}(\lambda_1)$ with $\lambda_1\ne 0$ is unbounded, the latter possibility forces $\mathcal{C}(0)$ to be unbounded as well.
\end{proof}

\begin{Remark}
The nondegeneracy of $0$ can be replaced by the requirement that, in \textnormal{(B6)}, the degree of $\nabla_{u}F$ has a nonzero $\bI$–coefficient for all $\lambda\in\bR$; this is the only property used in the proof.

If, however, this coefficient is $0$, one must impose assumptions substantially stronger than \textnormal{(E)} on the spectral properties of the matrix $A$ to prevent cancellations among combinations of products that reproduce the effect of the factor $\chi_{\bT}(\bT/H_{*}^{+})$. In the proof above we relied on the fact that $\chi_{\bT}(\bT/H_{*}^{+})$ appears in exactly one factor and is multiplied by the $\bI$–coefficient coming from the other two factors. This mechanism breaks down in general when the $\bI$–coefficient vanishes.
\end{Remark}

\begin{Remark}
Theorem \ref{thm:unbounded} extends the unboundedness results of \cite{RybickiLB}, \cite{RybSte2015} (spheres) and \cite{Stefaniak2025} (a symmetric space), where the linearization was restricted by the assumption $A=\mathrm{Id}$ (cf. \textnormal{(B5)}), hence admitting only a single eigenvalue. Here we remove that restriction and allow $A$ with multiple eigenvalues, while adopting the convenient cooperativity condition (all Laplace--Beltrami terms have the same sign). Under this alternative structural assumption and mild, verifiable assumptions, unboundedness follows.
\end{Remark}

\end{document}